\documentclass[11pt,a4paper]{article}

\usepackage{amsmath}
\usepackage{amsfonts}
\usepackage{amssymb}
\usepackage{amsthm}
\usepackage{graphicx}

\usepackage{macros}

\title{Random walks on quasirandom graphs}
\author{Ben Barber
        \footnote{Department of Pure Mathematics and Mathematical Statistics,
                  Centre for Mathematical Sciences,
                  Wilberforce Road, Cambridge, CB3 0WB, UK.
                  {\tt b.a.barber@dpmms.cam.ac.uk}}
        \and
        Eoin Long
        \footnote{School of Mathematical Sciences,
                  Queen Mary University of London,
                  UK.
                  {\tt e.p.long@qmul.ac.uk}}
       }

\begin{document}

\maketitle

\begin{abstract}
Let $G$ be a quasirandom graph on $n$ vertices, and let $W$ be a random walk on $G$ of length $\length n^2$. Must the set of edges traversed by $W$ form a quasirandom graph? This question was
asked by B\"ottcher, Hladk\'y, Piguet and Taraz. Our aim in this paper is to give a positive answer to this
question. We also prove a similar result for random embeddings of trees.
\end{abstract}

\section{Introduction}
\label{sec:intro}

Given a graph $G$ and sets $A, B \subseteq V(G)$, let
\[
 e_G(A,B) = \left|\left\{(a,b)\in A \times B: ab\in E(G)\right\}\right|
\]
be the number of edges from $A$ to $B$.
A graph $G$ with $n$ vertices and $\rho \binom {n}{2}$ edges is
\emph{$\epsilon$-quasirandom} if
\begin{equation*}
 |e_G(A,B) - \rho |A||B|| < \epsilon |A||B|
\end{equation*}
for all sets $A$, $B \subseteq V(G)$ with $|A|$, $|B| \geq \epsilon n$. Thus a
quasirandom graph resembles a random graph with the same density,
provided we do not look too closely. Quasirandom graphs were introduced by
Thomason \cite{Thomason} and have come to play a central role in probabilistic
and extremal graph theory. The reader is referred to the excellent
survey article by Krivelevich and Sudakov \cite{Krivelevich-Sudakov} for further
details.

The random graph $G_{n, p}$, in which edges appear independently with
probability $p$, is quasirandom with high probability.  More generally, given a
quasirandom graph $G$ we can, with high probability, obtain a new quasirandom
graph $\GP$ by retaining edges of $G$ with some fixed probability $p$.  (The random graph $G_{n,p}$ can be thought of as the result of applying this process
to the complete graph $K_n$.)

Another natural way to choose a random set of edges from an $n$-vertex graph $G$
is given by the following process. A random walk $W$ on $G$ is a sequence of
vertices $W_0, W_1, \ldots, W_l$ where $W_0$ is chosen from some initial
distribution and $W_{i+1}$ is selected uniformly from the neighbours of $W_i$,
with all choices made independently. We will be interested in the case when $W$
traverses some constant fraction of the edges of $G$, so we take the length $l$
of $W$ to be $\length n^2$ for some constant $\length > 0$. To avoid confusion with
the random walk $W$, let $\GW$ denote the random subgraph of $G$ consisting of
those edges traversed by $W$. The main question that we are interested in here 
is the following: given a quasirandom graph $G$, is it true (as with $\GP$) 
that $\GW$ is also quasirandom with high probability?

First note that this is true when this process is applied to the complete graph 
$G=K_n$. Indeed, $\GW$ is very close to $G_{n, p}$ for
some $p$.  This is because the sequence $W_0, W_1, \ldots$ is very nearly a
sequence of independent random vertices of $G$ (`very nearly' because
consecutive terms of the sequence are forced to be distinct).  Then $W_0W_1,
W_2W_3, \ldots$ and $W_1W_2, W_3W_4, \ldots$ are very nearly two sequences of
independent random edges of $G$, so $\GW$ is very close to a random subgraph of
$G$.

The following heuristic suggests that $\GW $ should also be quasirandom for
a general quasirandom graph $G$.
\begin{enumerate}
 \item[(1)] The graph $G$ is approximately regular, of degree $\rho (n-1)$, so the
equilibrium distribution of $W$ is approximately uniform.  The random walk $W$
`mixes rapidly', so most sequence terms have distributions close to the
equilibrium distribution, and $W$ visits each vertex around the same number of
times: approximately $\length n$ times.
 \item[(2)] For each vertex $v$ of $G$, the random walk leaves $v$ $\length n$ times,
so picks up a random set (chosen with replacement) of $\length n$ of the $\rho n$
edges at $v$.  Taking the union of these sets of edges gives a random subgraph
of $G$.
\end{enumerate}

While this simple plan seems quite plausible, there are two reasons why it is 
not easy to implement, both connected to (1) above. The first is 
related to what exactly it means to say that $W$ `mixes rapidly'.  Since 
quasirandomness does not say anything about small parts of the graph, $G$ might have small 
configurations of low degree vertices that can trap the random walk for long 
periods of time.  The second is that, even if we know the distribution of each 
$W_i$, these random variables are not independent for different $i$.

These difficulties can, however, be overcome, and an argument of the above form
can be used to show that the subgraph of $G$ spanned by $W$ will be quasirandom with high probability.

\begin{theorem} \label{thm:general-case}
 Given $\length, \rho, \eta >0$ there exists $\epsilon > 0$
 such that the following holds. Let $G$ be an $n$-vertex $\epsilon$-quasirandom
graph with $\rho \binom n 2$ edges, and let $W$ be a random walk on $G$ of length
$\length n^2$ starting at any vertex $W_0$ of $G$ with degree in
 $[(\rho - \epsilon)n,(\rho + \epsilon )n]$. Then, with  probability $1 - \Oe$,
the graph $\GW$ is $\eta $-quasirandom
 with $(1-e^{-2\length /\rho } + \Oe)\rho \binom n 2$ edges.
\end{theorem}

Here $o_\epsilon(1)$ means a quantity which is less than $f(\epsilon)$ for $n$
sufficiently large, for some $f(\epsilon)$ tending to zero with $\epsilon$. 
This dependence on $\epsilon$, and not just $n$, is necessary to deal with the
`bad small subgraph' problem described above.

If we also have a lower bound on the minimum degree of $G$ then we can say much
more: if we start with a graph that is $\epsilon$-quasirandom, then, with high
probability, $\GW$ will also be $\epsilon$-quasirandom.

\begin{theorem} \label{thm:bounded-minimum-degree}
 Let $\length, \epsilon , \rho  >0$ and let $\gamma
= C \epsilon ^{1/4}$ for some absolute constant
 $C>0$.  Let $G$ be an $n$-vertex $\epsilon $-quasirandom graph with $\rho
\binom n 2$ edges and minimum degree at least $\gamma n$, and let $W$ be a
random walk on $G$ of length $\length n^2$. Then, with probability $1 - o(1)$, the
graph $\GW$ is $\epsilon$-quasirandom with
 $(1-e^{-2\length /\rho } + o(1)) \rho \binom n 2$ edges.
\end{theorem}

In \sect{sec:list} we define an explicit model for our random walks and show that step (2) above works straightforwardly in this setting.  In \sect{sec:bounded-mindeg} we carry out (1) for
the case where we have a bound on the minimum degree of $G$.  This proves \thm{thm:bounded-minimum-degree},
and illustrates why we get the weaker conclusion in \thm{thm:general-case}.  In \sect{sec:general-case} we
use a more elaborate argument to perform (1) in the general case, proving
\thm{thm:general-case}.

The problems considered in this paper were suggested by B\"ottcher, Hladk\'y,
Piguet and Taraz \cite{Hladky} after they encountered similar problems in connection with their
work on tree packing.  Suppose that we are trying to pack many trees into a copy
of $K_n$.  One approach is to embed some of the trees randomly.  If we succeed
in packing a small number of trees, then it would be good to know that the
subgraph consisting of unused edges has nice enough properties that we can
iterate the argument and therefore pack a much larger number of trees.  If $H$
is a subgraph of $G$, and both graphs are quasirandom, then $G-H$ is also
quasirandom.  So it would be useful to have a result like \thm{thm:general-case}, but for random
images of trees rather than paths.  We consider such a generalisation in \sect{sec:trees}.

Since we will only prove asymptotic results we make a number of simplifying assumptions.  We assume that $\epsilon$ is sufficiently small compared to the other parameters, and are only interested in statements for $n$ sufficiently large.  We omit notation indicating the taking of integer parts, and ignore questions of divisibility when breaking walks into pieces of a given size.

\section{The list model}
\label{sec:list}

We now define a third model of a random subgraph to act as a staging post
between $\GW$ and $\GP$.  The subgraph $\GL{\entries}$ of $G$ is obtained by
selecting $\entries d(v)$ edges at each vertex $v$ of $G$ to be retained, with
all choices made independently and with replacement.  We give a rather elaborate
formal definition in order to introduce some ideas which will be useful later.

For each $v \in V(G)$, let $L_v$ be an infinite list of uniform selections from
the neighbourhood of $v$, with all choices made independently.  The entry $u$ on
the list $L_v$ corresponds naturally to the edge $uv$ of $G$, and we define
\[
 \GL{\entries} = \bigcup_{v \in V(G)} \{uv : u \text{ appears in the first
$\entries d(v)$ entries of } L_v \}.
\]

In this section we will show that $\GL{\entries}$ is very close to $\GP$ for
some $p$, in the sense that large subgraphs have similar densities in each
model.  It will then follow that $\GL{\entries}$ is quasirandom with high
probability.

We first calculate the expected density of $\GL{\entries}$ in $G$.  The reader is encouraged to focus on the case where we have a bound on the minimum degree.

\begin{lemma} \label{lem:subgraph-density}
 Let $G$ be an $\epsilon$-quasirandom graph on $n$ vertices with $\density
\binom n 2$ edges, and let $A, B \subseteq V(G)$ with $|A|, |B| \geq
\epsilon^{0.99} n$.  Then
 \[
  \E{e_{\GL{\entries}}(A,B)} = (1 - e^{-2\entries} + \Oe)e_G(A,B).
 \]
 Moreover, if the minimum degree of $G$ is at least $\mindeg n$, then, for all
$A, B \subseteq V(G)$,
 \[
  \E{e_{\GL{\entries}}(A,B)} = (1 - e^{-2\entries} + o(1))e_G(A,B).
 \]

\end{lemma}

The exact lower bound on the sizes of $A$ and $B$ in the general case is unimportant; any value asymptotically larger than $\epsilon$ would work equally well.

\begin{proof}
 Write $S = \{v \in V(G): d(v) \geq (\density - \epsilon) n \}$.  Then
 \[
  |e_G(V(G),S^c) - \density n |S^c|| > \epsilon n |S^c|,
 \]
 so, by $\epsilon$-quasirandomness, $|S^c| < \epsilon n$.

 The edge $uv$ of $G$ appears in $\GL{\entries}$ if and only if $u$ appears in
the first $\entries d(v)$ entries of $L_v$, or $v$ appears in the first
$\entries d(u)$ elements of $L_u$.  If $u, v \in S$, then the probability of
this occurring is
 \begin{equation*}
  1 - (1 - 1/d(v))^{\entries d(v)}(1 - 1/d(u))^{\entries d(u)} = 1 -
e^{-2\entries} + o(1),
 \end{equation*}
 since $d(v), d(u) \geq (\density - \epsilon) n$.  So
 \begin{align*}
  \E{e_{\GL{\entries}}(A,B)} & = (1 - e^{-2\entries} + o(1))e_G(A,B) +
O(\epsilon n (|A|+|B|)) \\
                             & = (1 - e^{-2\entries} + \Oe)e_G(A,B),
 \end{align*}
 since $e_G(A,B) \geq (\density - \epsilon) |A||B|$ and $|A|, |B| \geq
\epsilon^{0.99} n$.

 If $d(v) \geq \mindeg n$ for every $v \in V(G)$, then the probability of being
retained is $1 - e^{-2\entries} + o(1)$ for every edge of $G$, so
 \[
  \E{e_{\GL{\entries}}(A,B)} = (1 - e^{-2\entries} + o(1))e_G(A,B). \qedhere
 \]
\end{proof}

To show that the number of edges retained in any subgraph is close to its
expectation we use Talagrand's concentration inequality 
\cite{Talagrand}. In its usual form Talagrand's inequality is asymmetric and
bounds a random variable in terms of its median. We
use the following symmetric version (see \cite[Chapter 20]{Molloy-Reed}) that
gives concentration of the random variable about its mean.\footnote{We note that the proof in \cite{Molloy-Reed} is incorrect, as a simplified form of the failure probability valid for a small deviation $t$ is used when $t$ is large.  However, when $t$ is large, a different simplification can be used instead at the cost of (a) increasing the constants slightly and (b) imposing a very mild additional condition on the random variable $X$.}

\begin{theorem}[Talagrand's inequality] \label{thm:talagrand}
Let $\Omega = \prod _{i=1} ^N \Omega _i$ be a product of probability spaces with
the product measure.  Let $X$ be a random variable on 
$\Omega$ such that
\begin{enumerate}
 \item[(i)] $|X(\omega) - X(\omega')| \leq c$ whenever $\omega$ and $\omega'$
differ on only a single coordinate for some constant $c>0$;
 \item[(ii)] whenever $X(\omega) \geq r$ there is a set $I \subseteq \{1,
\dotsc, N\}$ with $|I| = r$ such that $X(\omega') \geq r$ for all
$\omega'\in\Omega$ with $\omega_i' = \omega_i$ for all $i \in I$;
\end{enumerate}
 and suppose that the median of $X$ is at least $100c^2$.  Then, for $0 \leq s \leq \E{X}$,
\begin{equation*}
 \Pr{|X - \E{X}| \geq s + 90c\sqrt{\E{X}}} \leq 4 e^{- s^2 / 8 c^2\E{X}}.
\end{equation*}
\end{theorem}

\begin{lemma} \label{lem:subgraph-concentration}
 Let $G$ be an $n$-vertex $\epsilon$-quasirandom graph, and fix $\entries > 0$. 
Then, with probability $1-o(1)$,
 \[
  e_{\GL{\entries}}(A,B) = (1 - e^{-2\entries} + \Oe)e_G(A,B),
 \]
 for all  $A, B \subseteq V(G)$ with $|A|, |B| \geq \epsilon^{0.99} n$.
Moreover, if the minimum degree of $G$ is at least $\mindeg n$, then the same
result holds for $|A|, |B| \geq \epsilon n$, with $\Oe$ replaced by $o(1)$.
\end{lemma}

\begin{proof} 
We apply \thm{thm:talagrand} to the space $\Omega = \prod_{v \in V(G)}
\prod_{i=1}^{\entries d(v)} N(v)$, where each
neighbourhood has the uniform probability measure; we can view $\Omega$ as the
space of choices for the first $\entries d(v)$ entries of each
list $L_v$.  For $A, B \subseteq V(G)$ with $|A|,|B| \geq \epsilon^{0.99} n$,
let $X_{A,B} = e_{\GL{\entries}}(A,B)$. It is easy to
see that $X_{A,B}$ satisfies the conditions of Talagrand's inequality. Indeed,
(i) holds since changing a list entry can
change $X_{A,B}$ by at most $c=2$. Furthermore, (ii) holds since, if $X_{A,B}
\geq s$, then there are $s$ list entries
witnessing this fact.  Finally, the technical condition on the median holds for large $n$, since the probability that $X_{A,B}$ is less than any constant is $o(1)$.  Therefore, by \thm{thm:talagrand}, for
$180\sqrt{\E{X_{A,B}}} \leq t \leq \E{X_{A,B}}$ and large $n$
we have
\begin{align*}
 \Pr{\abs{X_{A,B} - \E{X_{A,B}}} \geq 2t} \leq 4 e^{-t^2/ 32\E{X_{A,B}}}.
\end{align*}
By \lem{lem:subgraph-density} we have $\E{X_{A,B}} = (1 - e^{-2\entries } +
\Oe)e_G(A,B)$.  Since $e_G(A,B) \geq (\density - \epsilon) \epsilon^{1.98}n^2$, taking
$t = C' \sqrt {n\E{X_{A,B}}}$ $(=o(\E{X_{A,B}}))$ for large enough $C'>0$ gives
that
\begin{equation*}
 \Pr{\abs{X_{A,B} - (1 - e^{- 2 \entries} + \Oe)e_G(A,B)} \geq 2t} \leq 8^{-n}.
\end{equation*}
But there are at most $2^n$ choices for $A$ and $2^n$ choices for $B$.
Therefore, with probability at least 
$1 - 2^{-n}$, we have that $X_{A,B} = (1 - e^{- 2 \entries} + \Oe)e_G(A,B)$, for
all pairs $(A,B)$ with $|A|,|B| \geq \epsilon^{0.99} n$.  The `moreover'
statement is proved identically, using the `moreover' statement from
\lem{lem:subgraph-density}.
\end{proof}

This is enough to ensure that $\GL{\entries}$ is quasirandom with high
probability.

\begin{theorem} \label{thm:list-quasirandomness}
 Let $\entries, \mindeg > 0$ and let $G$ be an $n$-vertex $\epsilon$-quasirandom graph with $\density \binom n 2$ edges.  Then, with probability $1-o(1)$, $\GL{\entries}$
is $\Oe$-quasirandom.  Moreover, if the minimum degree of $G$ is at least
$\mindeg n$, then, with probability $1-o(1)$, $\GL{\entries}$ is
$\epsilon$-quasirandom.
\end{theorem}

\begin{proof}
 By \lem{lem:subgraph-density}, with probability $1-o(1)$,
 \[
  |e_{\GL{\entries}}(A,B) - (1 - e^{-2\entries})e_G(A,B)| = \Oe e_G(A,B),
 \]
 for all $A, B$ with $|A|, |B| \geq \epsilon^{0.99} n$.  By the definition of
quasirandomness, we also have that
 \[
  |e_G(A,B) - \density |A||B|| < \epsilon |A||B|.
 \]
 So by the triangle inequality,
 \begin{align*}
  |e_{\GL{\entries}}(A,B) - (1 - e^{-2\entries}) \density |A||B|| & < \Oe
e_G(A,B) + (1 - e^{-2\entries}) \epsilon |A||B| \\
                                                                  & \leq (\Oe +
(1 - e^{-2\entries}) \epsilon) |A||B| \\
                                                                  & = \Oe
|A||B|.
 \end{align*}
 Hence $\GL{\entries}$ is $\delta$-quasirandom with $\delta = \max (\epsilon^{0.99},
\Oe)$, where the $\Oe$ is taken from the last line.  Since $(1 - e^{-2\entries})
\epsilon < \epsilon$, the `moreover' statement follows from the `moreover'
statement of \lem{lem:subgraph-density}.
\end{proof}

Having shown that $\GL{\entries}$ is quasirandom with high probability, it
suffices to show that $\GW$ is close to $\GL{\entries}$ for some $\entries$. 
The construction of the random walk $W$ requires, at each visit to a vertex $v$,
a choice of a random neighbour of $v$.  We obtain a coupling of $\GW$ and
$\GL{\entries}$ by, at the $j$th visit to $v$, taking this choice to be the
$j$th entry of the list $L_v$.  Then $\GW$ and $\GL{\entries}$ both consist of
the edges corresponding to some initial segments of the lists $L_v$, and it is
enough to show that we can choose $\entries$ such that the lengths of these
initial segments are similar: that is, that the number of times the random walk
$W$ visits each vertex of $G$ is roughly proportional to its degree.

We give two arguments.  The first, appearing in \sect{sec:bounded-mindeg},
applies when we have a good lower bound on the minimum degree of $G$.  The
second, appearing in \sect{sec:general-case}, applies to a general quasirandom
graph $G$, but necessarily gives a weaker result.  We include the argument for
the special case where $G$ has large minimum degree for two reasons.  First, it
proves the stronger result that there is no loss of quasirandomness
when we pass from $G$ to $\GW$, which could be useful for some applications. 
Second, it illustrates why the natural approach cannot work in general,
justifying the use of a more technical argument in \sect{sec:general-case}.

\section{Bounded minimum degree}
\label{sec:bounded-mindeg}

To begin this section we recall some useful facts. A random walk $W$ on a graph $G$ is
a Markov chain with transition matrix $P$ given by
\begin{equation*}
 P_{uv} = \begin{cases}
           1/d(u) & \text{if } uv \in E(G); \\
           0      & \text{if } uv \not\in E(G).
          \end{cases}
\end{equation*}
Thus $P$ is a normalised version of the adjacency matrix $A$, where each row has been scaled by the 
degree of the corresponding vertex. The eigenvalues of $P$ are all real;
let these be $\lambda_1 \geq \lambda_2 \geq \cdots \geq \lambda_n$ and write $\lambda =
\max(\abs{\lambda_2},\abs{\lambda_n})$. The first eigenvalue $\lambda_1$ of $P$ is always equal 
to $1$ and has a corresponding eigenvector $\pi = (\pi _v)$ given by $\pi_v = \frac{d(v)}{2e(G)}$. 
This vector $\pi$ is called the \defn{stationary distribution} of the walk $W$. 
If $G$ is disconnected, then $\lambda_2 = 1$;  if $G$ is bipartite, then $\lambda_n = -1$.  If $G$ is connected and non-bipartite, then $\lambda < 1$ and the walk $W$ is irreducible and aperiodic.  Later in this section, $G$ will be $\epsilon$-quasirandom and have minimum degree at least $\gamma n$ with $\gamma > \epsilon$.  It will therefore be both connected (as the minimum degree condition forces every component to have size greater than $\epsilon n$, and components have no edges between them) and non-bipartite (else the larger class would have no edges to itself).  In this case it is well-known (for  
example, see \cite[Theorem 4.9]{Levin-Peres-Wilmer}) 
that, for any initial distribution of $W_0$, 
the distribution of $W_i$ converges to $\pi$ as $i \to \infty$ (i.e. $\Pr{W_i = v} \to \pi_v$ as $i \to \infty$ for each
$v$).

The following standard result, which can be read out of Jerrum and Sinclair
\cite{Jerrum-Sinclair}, gives control on the \emph{rate} of this convergence.

\begin{lemma}\label{lem:exponential-decay}
 For any $n$-vertex graph $G$ with minimum degree at least $\mindeg n$, and any initial distribution of $W_0$, we
 have
\begin{equation*}
 \max_{v \in V(G)} |\Pr{W_i = v} - \pi_v| \leq c_\mindeg \lambda^i,
\end{equation*}
 for some $c_\mindeg$ depending on $\mindeg$.
\end{lemma}

If $G$ is a regular $\epsilon$-quasirandom graph then 
$\lambda$ is small on the scale of $\epsilon$.  (This is because the `spectral gap' of a quasirandom graph is large \cite{Chung-Graham-Wilson}, and $P$ is a scalar multiple of $A$ when $G$ is regular.)  For a general $\epsilon$-quasirandom graph
this need not be true: for example, if $G$ contains a 
small connected component, then $\lambda = 1$ (the $1$-eigenspace is spanned 
by the stationary distributions of each connected component of $G$). Similarly, $\lambda$ 
can be very close to $1$ if there is a small set of vertices that is only weakly connected to the rest 
of the graph. However, a lower bound on the minimum degree of $G$ is enough to recover an upper bound on $\lambda$.

\begin{lemma}\label{lem:spectral-gap}
 Let $G$ be an $n$-vertex $\epsilon$-quasirandom graph with $\density \binom n 2$ edges and minimum degree at least $\mindeg n$, where
 $\mindeg \geq C\epsilon^{1/4}$ for some absolute constant $C>0$.  Then, for $n$
 sufficiently large, $\lambda \leq 1/2$.
\end{lemma}

Before proving \lem{lem:spectral-gap}, we note the following simple observation about quasirandom graphs which we
will use repeatedly. 

\begin{proposition} \label{prop:balanced}
 Let $G$ be an $n$-vertex $\epsilon$-quasirandom graph with $\density \binom n 2$ edges, and let $X$ be a set of vertices with $|X| \geq \epsilon n$.
 Let $Y = \{v \in V(G) : \abs{e_G(v, X) - \density |X|} \geq \epsilon |X|\}$.  Then 
 $|Y| < 2 \epsilon n$.
\end{proposition}

\begin{proof}
 We have $Y = Y^+\cup Y^-$ where
 \begin{align*}
  Y^+ & = \{v \in V(G) : e_G(v, X) \geq \density |X| + \epsilon |X|\}, \\
  Y^- & = \{v \in V(G) : e_G(v, X) \leq \density |X| - \epsilon |X|\}.
 \end{align*}
 Clearly
 \begin{gather*}
  \abs{e_G(X, Y^+) - \density \left|X\right| \left|Y^+\right|} \geq \epsilon \left|X\right| \left|Y^+\right|, \\
  \abs{e_G(X, Y^-) - \density \left|X\right| \left|Y^-\right|} \geq \epsilon \left|X\right| \left|Y^-\right|.
 \end{gather*}
 But then, since $G$ is $\epsilon$-quasirandom and $|X| \geq \epsilon n$, we must have $|Y^+|, |Y^-| < \epsilon n$.
\end{proof}

In particular, taking $X = V(G)$ there are at least $(1-2\epsilon)n$ vertices $v$ of $G$ with $|d(v) - \rho n| \leq \epsilon n$. 
We will call such vertices \defn{balanced}.

\begin{proof}[Proof of \lem{lem:spectral-gap}]
The proof follows a well-known argument (see for example \cite{Chung-Graham-Wilson}). We first estimate the number of labelled 
copies of $C_4$ in $G$, and then evaluate the trace of $P^4$ in two different ways.  Note that the implicit constants in our use of $O(\cdot)$ notation here are absolute (i.e. independent of $\epsilon$, $\density$, $\mindeg$ and $\lambda$).

The number of labelled copies of $C_4$ in $G$ is
\[
 C_4(G) = 2\sum_{u \in V(G)} \sum_{v \in V(G)\setminus\{u\}}\binom{|N(u)\cap N(v)|}{2}.
\]
If $u$ is balanced, then, by \prop{prop:balanced}, the number of $v$ with $|N(u) \cap N(v)|$ far from $\density^2 n$ is small, so
\begin{align*}
 C_4(G) & = 2 \cdot (1 + O(\epsilon))n \cdot (1 + O(\epsilon))n \cdot \binom{(\density +O(\epsilon))^2n}{2} + O(\epsilon) n^2 \binom n 2
\\      & = (\density + O(\epsilon))^4n^4 + O(\epsilon) n^4
\\      & = \left(1 + \order{\epsilon/\density ^4}\right)\density^4n^4,
\end{align*}
where the main term here accounts for balanced vertices $u$ and $v$ with close to $\density ^2n$ common neighbours, and the error 
term bounds the contribution to the sum from each other pair by $\binom n 2$.

Now the trace of $P^4$ is a weighted sum of the closed walks of length $4$ in $G$, where the weight of the closed walk  $uvwx$ is 
$1/(d(u)d(v)d(w)d(x))$.  Thus
\begin{align*}
  \sum_{v \in V(G)} (P^4)_{vv}
   & = \frac{(1 + O(\epsilon/\density ^4))\density^4n^4}{((\density + O(\epsilon ))n)^4}
       + \frac{O(\epsilon) n^4}{(\mindeg n)^4}
       + \frac{\order{n^3}}{(\mindeg n)^4} \\
   & = 1 + \order{\epsilon / \density ^4} + \order{\epsilon / \mindeg^4} + \order{1 / (\mindeg^4n)},
\end{align*}
where the main term counts the contribution from $4$-cycles containing only balanced vertices and the error
terms account for the contributions from $4$-cycles with at least one unbalanced vertex and
from closed walks of length $4$ which are not $4$-cycles respectively.  (The lower bound on the minimum degree 
of $G$ gives an upper bound of $1/(\mindeg n)^4$ for the weight of any one walk.)  But we also have
\begin{equation*}
 \sum_{v \in V(G)} (P^4)_{vv} = \sum_{i=1}^n \lambda_i^4 = 1 + \sum_{i=2}^n \lambda_i^4,
\end{equation*}
from which it follows that
\begin{equation*}
  \lambda^4 \leq \sum_{i=2}^n \lambda_i^4 = \order{\epsilon / \density ^4} + \order{\epsilon / \mindeg^4} + \order{1 / (\mindeg^4n)} \leq 1/16,
\end{equation*}
for $\density \geq \mindeg \geq C \epsilon^{1/4}$ and $n$ sufficiently large.
\end{proof}

For the next lemma we will need to approximate one probability measure by another on the same space. Given a finite probability space
$\Omega$, the \defn{total variation distance} between two probability measures $\mu_1$ and $\mu_2$ is defined by
\begin{equation*}
 d_{TV}(\mu_1, \mu_2) = \frac{1}{2}\sum_{\omega \in \Omega} |\mu_1(\omega) - \mu_2(\omega)|.
\end{equation*}
This is the amount of probability mass that would have to be moved to turn one distribution into the other.

Combining \lem{lem:exponential-decay} with \lem{lem:spectral-gap}, it is easy to see that the total variation distance between $W_t$ 
and a vertex sampled from the stationary distribution is small when $t$ is moderately large.  In fact, we get much 
more. 

Let $L = (\log n)^2$, and let $K = \length n^2 /L$. Given $i<L$, let $W^{(i)}$ denote the
subsequence of $W$ obtained by starting from $W_i$ and taking $L$ steps at a time: that is, $W^{(i)} = (W^{(i)}_1, \ldots, W^{(i)}_{K})$
where $W^{(i)}_{j} = W_{i+(j-1)L}$ for all $j\leq K$.  For each $v\in V(G)$, let $X^{(i)}_{v}$ be the random variable
which counts the number of times $W^{(i)}$ visits $v$. Our next lemma shows that with high probability $X^{(i)}_{v}$ is close to its 
mean.

\begin{lemma} \label{lem:subsequence}
  Let $G$ be a graph satisfying the conditions of \lem{lem:spectral-gap} and let $v \in V(G)$. Then we have
\begin{equation*}
 \Pr{\abs{X^{(i)}_v - K \pi_v} \geq \sqrt {\frac{8 \log n} {K\pi_v} } K\pi_v} = \order{n^{-3}}.
\end{equation*}
\end{lemma}
\begin{proof}
 Let $\mu = \pi^K$ be the $K$-fold product measure of $\pi$ on $V(G)^K$; that is,
 $\mu (w) = \prod _{i=1}^K\pi _{w_i}$ for $w \in V(G)^K$. By
 \lem{lem:exponential-decay} (absorbing the constant $c_\mindeg$ into the $O(\cdot)$ notation) and \lem{lem:spectral-gap}, we have
	\begin{align*}
    \Pr{W^{(i)} = w} & = \mathbb P\big(W^{(i)}_1 = w_1\big)\mathbb P\big(W^{(i)}_2 = w_2 | W^{(i)}_1 = w_1\big) \dotsm \mathbb P\big(W^{(i)}_{K} = w_{K} | W^{(i)}_{K - 1} = w_{K - 1}\big) \\
                & = \left(\pi_{w_1} + \order{2^{-(\log n)^2}}\right)\left(\pi_{w_2} + \order{2^{-(\log n)^2}} \right) \dotsm \left(\pi_{w_{K}} + \order{2^{-(\log n)^2}} \right) \\
                & = \left(\pi_{w_1} + \order{n^{-6}}\right)\left(\pi_{w_2} + \order{n^{-6}} \right) \dotsm \left(\pi_{w_{K}} + \order{n^{-6}} \right) \\
                & = \left(1 + \order{n^{-3}}\right)\mu (w),
  \end{align*}
 since $\frac \mindeg {\density n} \leq \pi_v \leq \frac 1 {\density n}$ for all $v$ and $K = O(n^2)$. Summing over all $w$ gives that
\begin{equation*}
d_{TV}\left({\mathbb P} ,\mu \right) = \order{n^{-3}},
\end{equation*}
 where $\mathbb{P}$ is the measure on $V(G)^K$ induced by $W^{(i)}$.
 Now let
 \begin{equation*}
  A = \left\{w \in V(G)^K: |X^{(i)}_v(w) - K \pi_v| \geq \sqrt {8 \log n  K\pi_v}\right\}.
 \end{equation*}
 By Chernoff's inequality (see \cite[A.1.11 and A.1.13]{Alon-Spencer}),
 \[
  \mu (A) \leq 2e^{- (4 + o(1)) \log n} = \order{n^{-3}}.
 \]
Since $\Pr{A} \leq \mu (A) + d_{TV}({\mathbb{P}},\mu )$, the result follows.
\end{proof}

Now let $X_v = \sum _{i=0}^{L-1} X^{(i)}_v$ be the number of visits $W$ makes to vertex $v$.
Observing that $LK\pi_v = \length n^2\pi_v = (1+\frac 1 {n-1})\frac {\length}{\rho} d(v)$, we obtain the following corollary by summing over $i$ and $v$.

\begin{corollary} \label{cor:bmd-visits}
Let $\length, \epsilon, \density, \mindeg> 0$ with $\density, \mindeg \geq C\epsilon^{1/4}$ for some absolute constant $C>0$.
Let $G$ be an $n$-vertex $\epsilon$-quasirandom graph with minimum degree at least $\mindeg n$, and let $W$ be a
random walk on $G$ of length $\length n^2$. Then
\begin{equation*}
\pushQED{\qed} 
  \Pr{\abs{X_v - \frac {\length}{\rho} d(v)} \geq \sqrt {\frac {8 \log n}{K\pi_v}} \frac {\length}{\rho} d(v) \text{ for some } v} = \order{n^{-1}}. \qedhere
\popQED
\end{equation*}
\end{corollary}

Since $K\pi_v = \Omega(n/(\log n)^2)$, with high probability the number of visits $W$ makes to each $v \in V(G)$ is $\left(\frac{\length}{\density} + o(1)\right) d(v)$.  We can now complete the proof of \thm{thm:bounded-minimum-degree}.

\begin{proof}[Proof of \thm{thm:bounded-minimum-degree}]
 By \cor{cor:bmd-visits}, we have that, with probability $1-o(1)$,
 \[
  \GL{ \length / \density - o(1)} \subseteq \GW \subseteq \GL{ \length / \density + o(1)}.
 \]
 From the proof of \thm{thm:list-quasirandomness}, we have that, with probability $1-o(1)$,
 \[
  |e_{\GW}(A,B) - (1 - e^{-2\frac \length \density}) \density |A||B|| < (1 - e^{-2\frac \length \density} + o(1)) \epsilon |A||B|,
\]
 for all $A, B \subseteq V(G)$ with $|A|, |B| \geq \epsilon n$.  Since $1 - e^{-2\frac \length \density} < 1$, $\GW$ is $\epsilon$-quasirandom with probability $1-o(1)$.
\end{proof}

\section{General case}
\label{sec:general-case}

We now move to the case of a general $\epsilon $-quasirandom graph $G$ with edge
density $\density $. Such $G$ must always contain
a connected component of order at least $(1-\epsilon )n$ (as otherwise we can
find two sets of size at least $\epsilon n$ with no edges between them), so by
restricting our walk to this component we can
assume that $G$ is connected.

The extra difficulty in the general case is that there might be small sets of
vertices that are only weakly connected to the rest of the graph in which the
random walk can get stuck.  For example, let $G$ be a graph consisting of a
small clique of order $\epsilon^2 n / 2$ joined to a large clique of order
$(1-\epsilon^2 / 2)n$ by a single edge. Then $G$ is $\epsilon$-quasirandom, but
it is not even true that the number of edges in $\GW$ is concentrated near some
value.  Indeed, if we start our random walk in the large clique then with
positive probability (depending on $\epsilon$ but not on $n$) $W$ will lie
entirely within the large clique, but there is also a positive probability
(depending on $\epsilon$ but not on $n$) that $W$ will cross to the small clique
in the first $\epsilon n^2$ steps and remain there.  So for general quasirandom
graphs we cannot hope for as strong a result as
\thm{thm:bounded-minimum-degree}, and our assertions about high probability will
necessarily depend on $\epsilon$ as well as $n$.  In this section we use `with
high probability' to mean `with probability $1-\Oe$', with $\Oe$ small
(depending on $\epsilon$) for large $n$ as defined in \sect{sec:intro}.

Our task in this section is to find a weaker replacement for
\cor{cor:bmd-visits} in \sect{sec:bounded-mindeg}. Instead of saying that the
random walk visits every vertex $v$ around $\frac \length \density d(v)$ times,
we ask instead that the random walk visits \emph{most} vertices of $G$ around
$\frac \length \density d(v)$ times.  Recall that we call a vertex $v$
\defn{balanced} if $|d(v)-\rho n|\leq \epsilon n$. We will show that, if $W$ is a
random walk of
length $\length n^2$ on $G$ with $W_0$ balanced, then, with high probability, $W$
hits
most vertices of $G$ about the right number of times.  The results in
\sect{sec:list} can then be used to prove
\thm{thm:general-case} in the same way that \thm{thm:bounded-minimum-degree}
was deduced from \cor{cor:bmd-visits}.

Our first lemma gives a lower bound on the probability that a given step of a
random walk $W$ is in a set $S \subseteq V(G)$.  
Write $\indicator{X}$ for the indicator function of a set $X$ and
$\indicator{v}$ for the indicator function of the set $\{v\}$.  Note
that if the initial distribution for $W_0$ is $\pi$ then $\Pr{W_i\in S} = \sum
_{v\in S} \pi_v = \pi \cdot \indicator S$ for
any set $S\subseteq V(G)$ when $i\geq 0$. The next result shows that this is
still almost true if $W$ starts from a balanced vertex,
$S$ is large and $i\geq 2$.

\begin{lemma} \label{bigsetprobabilities}
 Let $G$ be a connected $n$-vertex $\epsilon$-quasirandom graph with $\density
\binom {n}{2}$ edges, and let $v$ be a balanced vertex.  Let $S \subseteq V(G)$
with $|S| \geq \epsilon n$.  Then, for a random walk $W$ starting at $v$, we
have
\begin{equation*}
  \Pr{W_i \in S} \geq \pi \cdot \indicator S - 8 \sqrt \epsilon / \density \geq
|S|/n - 9 \sqrt \epsilon / \density,
\end{equation*}
for $i \geq 2$ and $n$ sufficiently large.
\end{lemma}

\begin{proof}
 We first show that the random walk is quite well mixed after only two steps. 
Let $A$ be the set of neighbours of $v$ with degree at most $(\density +
\epsilon) n$ and $B$ be the set of vertices with at least $(\density -
\epsilon)|A|$ neighbours in $A$; thus $A$ and $B$ are the `well-behaved' first and
second neighbourhoods of $v$.  By $\epsilon$-quasirandomness, $|A| \geq d(v) -
\epsilon n \geq (\density - 2\epsilon) n$ and $|B| \geq (1 - \epsilon) n$.  We
have
 \begin{equation*}
  \indicator{v}P = \frac{1}{d(v)}\indicator{N(v)} \geq \frac{1}{(\density +
\epsilon)n} \indicator{A},
 \end{equation*}
 where the inequality holds in each coordinate.  For $x \in B$,
 \begin{equation*}
  \left(\indicator{A}P\right)_x = \!\!\! \sum_{\substack{y \in A \\ xy \in
E(G)}} \!\! \frac{1}{d(y)}
                                \geq \frac {(\density - \epsilon)(\density -
2\epsilon)n} {(\density + \epsilon)n}
                                \geq \density(1 - 4\epsilon/\density),
 \end{equation*}
 where the first inequality holds since each $y\in A$ has degree at most
$(\density + \epsilon) n$, $x$ has at least $(\density - \epsilon)|A|$ neighbours in $A$
and $|A| \geq (\density - 2\epsilon) n$.  Since the entries of $P$ are
non-negative we can compose these inequalities to obtain
 \begin{equation*}
  \indicator{v}P^2 \geq \frac {\left(1 - 5\epsilon / \density\right)} {n}
\indicator{B}.
 \end{equation*}
 Let $\mathbf{b} = \frac {\left(1 - 5\epsilon / \density\right)} {n}
\indicator{B}$.  Since $\pi_x = \frac{d(x)}{2\density \binom n 2}$, if $x$ is a
balanced vertex then $\frac{(1-\epsilon/\density)}{n-1} \leq \pi_x \leq
\frac{(1+\epsilon/\density)}{n-1}$; otherwise we have the weaker bound $\pi_x
\leq \frac{1}{\density n}$.  Since at most $2\epsilon n$ vertices are unbalanced
and at most $\epsilon n$ vertices are not in $B$,
 \begin{equation*}
  \ltwo{\mathbf{b} - \pi} \leq {\left(n{\left(\frac{7\epsilon}{\density
n}\right)\!}^2 + 3\epsilon n{\left(\frac {2} {\density
n}\right)\!}^2\right)\!}^{1/2} \leq {\left( \frac {64 \epsilon} {\density^2 n}
\right)\!}^{1/2}.
 \end{equation*}
 Then, for $i \geq 2$,
 \begin{align*}
  \Pr{W_i \in S} & = \indicator{v} P^i \indicator{S} \\
                 & = \indicator v P^2 \cdot P^{i-2}\indicator S \\
                 & \geq \mathbf{b} P^{i-2}\indicator S \\
                 & = \pi P^{i-2} \indicator S + (\mathbf{b} - \pi) P^{i-2}
\indicator S.
 \end{align*}
 By Cauchy-Schwarz, and the fact that the eigenvalues of $P$ are at most 1 in absolute value,
 \begin{equation*}
  \|(\mathbf{b} - \pi) P^{i-2} \indicator S\|_2 \leq \ltwo{\mathbf{b} - \pi}
\ltwo{\indicator S} \leq{\left( \frac {64 \epsilon |S|} {\density^2 n}
\right)\!}^{1/2} \leq 8 \sqrt \epsilon / \density,
 \end{equation*}
 and so
 \begin{equation*}
  \Pr{W_i \in S} \geq \pi \cdot \indicator S - 8 \sqrt \epsilon / \density,
 \end{equation*}
 proving the first inequality.  Since at least $|S| - 2\epsilon n$ elements of
$S$ are balanced,
 \begin{equation*}
  \pi \cdot \indicator S = \sum_{x\in S} \frac {d(x)}{2 \density \binom n 2}
\geq \frac {(|S| - 2\epsilon n)(\density - \epsilon)}{\density n} \geq |S|/n -
2\epsilon - \epsilon/\density \geq |S| / n - \sqrt \epsilon / \density,
 \end{equation*}
 which proves the second inequality.
\end{proof}

We now consider the following variant of the list model for constructing a
random walk.  Fix some small length $L$ and let $K = \length n^2 /L$.  By a
\defn{block rooted at $v$} we mean a random walk of length $L$ starting at $v$. 
For each vertex $v$, let $\Lambda_v$ be an infinite list of blocks rooted at
$v$.  We construct a random walk of length $\length n^2$ as follows.  Choose
$W_0$ from the given initial distribution, and, at each stage $s=1,\dotsc,K$,
let $W_{(s-1)L}\cdots W_{sL}$ be the first unused block rooted at $W_{(s-1)L}$. 
At the end of the construction we have examined $K$ blocks in total from the top
of the $n$ lists.  Let $M$ be the set of blocks examined (equivalently, the
\emph{multiset} of roots of blocks used).

\begin{figure}[ht]
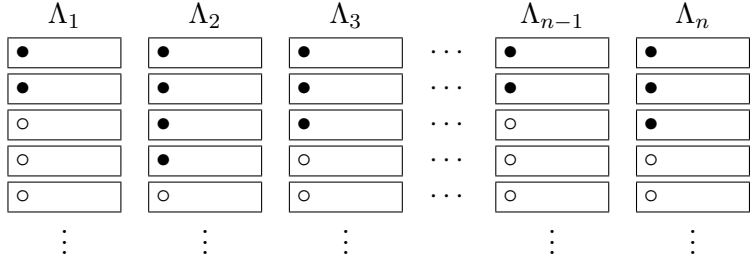


\newcommand{\bblock}{\framebox[1.5cm][l]{$\bullet$}}
\newcommand{\wblock}{\framebox[1.5cm][l]{$\circ$}}

\begin{equation*}
 \begin{array}{cccccc}
  \Lambda_1 & \Lambda_2 & \Lambda_3 &        & \Lambda_{n-1} & \Lambda_n    \\
  \bblock  & \bblock   & \bblock    & \cdots & \bblock       & \bblock \\
  \bblock  & \bblock   & \bblock    & \cdots & \bblock       & \bblock \\
  \wblock  & \bblock   & \bblock    & \cdots & \wblock       & \bblock \\
  \wblock  & \bblock   & \wblock    & \cdots & \wblock       & \wblock \\
  \wblock  & \wblock   & \wblock    & \cdots & \wblock       & \wblock \\
  \vdots   & \vdots    & \vdots     &        & \vdots        & \vdots
 \end{array}
\end{equation*}
\caption{The construction examines $K$ blocks from the top of the lists
$\Lambda_v$, but we cannot tell in advance which blocks these will be.}
\label{fig:block-list-model}
\end{figure}

This construction generalises the simple list model (which corresponds to the
case $L=1$), and we again hope to exploit the independence of blocks by applying
standard concentration inequalities.  There are two main obstacles.  One is that
we do not know anything about the distribution of a block rooted at a vertex $v$
which is not balanced.  We therefore first show that most of the root vertices
are balanced.  The second obstacle is that we do not know in advance which set
of blocks we will examine.  We handle this by approaching the problem from the
other direction: for a given multiset $M$, what is the probability that the
corresponding blocks do not contain an even distribution of the vertices?  This
turns out to be small enough that summing over all possible $M$ gives the bound
we require.

\begin{lemma} \label{hitting-large-balanced-set}
  Let $G$ be a connected $n$-vertex $\epsilon$-quasirandom graph with $\density
\binom {n}{2}$ edges, and let $W$ be a random walk of length $\length n^2$
starting at a balanced vertex of $G$.  Let $\delta = 3 \sqrt[4] \epsilon / \sqrt
\density$ and suppose that $n$ is sufficiently large.  Then with probability at
least $1 - 3\delta$ there exists a set $B\subseteq V(G)$, with $|B| \geq (1 -
\delta )n$, such that each vertex in $B$ is hit at least $(1 - 4\delta)\length
n$ times by $W$.
\end{lemma}

\begin{proof}
 Take $L = \omega_n$ for any $\omega_n \ll n / \log n$ which tends to infinity
as $n\to\infty$, and let $K = \length n^2 /L$.  Construct a random walk $W$ as
described above and let $x_1, \dotsc, x_K$ be the roots of the $K$ blocks used. 
We first show that with high probability many of the vertices $\{x_1,\ldots
,x_K\}$ are balanced.

 Let $U$ be the number of $x_i$ that are unbalanced.  The start vertex $x_1$ is always balanced by assumption.  By
\lem{bigsetprobabilities}, for $i \geq 2$,
 \begin{equation*}
  \Pr{x_i\text { is unbalanced}} \leq 1 - \left((1 - 2\epsilon) - \delta
^2\right) \leq 2\delta ^2,
 \end{equation*}
 since there are at least $(1-2\epsilon )n$ balanced vertices and $\delta^2 >
2\epsilon$.  By Markov's inequality,
 \begin{equation*} 
  \Pr{U \geq \delta K} \leq \frac{\E{U}}{\delta K} \leq \frac{2 \delta^2 K}
{\delta K} = 2 \delta,
 \end{equation*}
so with probability at least $1-2\delta$, at least $(1-\delta)K$ blocks are used whose starting vertex is balanced.  We now show that, with high probability, for any such multiset of balanced vertices $M$, the corresponding blocks contain most balanced vertices about the right number of times.

So let $M$ be a fixed multiset of $(1-\delta)K$ balanced vertices and let $W^{(1)},
W^{(2)}, \dotsc,$ $W^{((1-\delta)K)}$ be the corresponding blocks.  Note that, as $M$ is fixed, the blocks $W^{(i)}$ are independent.  Let $S\subseteq V(G)$ with $|S|\geq\delta n$.  By \lem{bigsetprobabilities},
for every $1 \leq i \leq (1-\delta )K$ and every $j \geq 2$ we have
$\Pr{W^{(i)}_j \in S} \geq \delta - \delta^2$.  Let $X_{ij}$ be the indicator of
the event $W^{(i)}_j \in S$, let $X_j = \sum_{i=1}^{K} X_{ij}$ and let $X_{M,S}
= \sum_{j=1}^L X_j$.  For fixed $j$ the $X_{ij}$ are independent, so by
Chernoff's inequality (see \cite[Appendix A]{Alon-Spencer}),
 \begin{equation*}
  \Pr{X_j < (\delta - 2 \delta^2)|M|} \leq e^{- 2 \delta ^4 |M|}.
 \end{equation*}
 Hence
 \begin{align*}
  \Pr{X_{M,S} < (\delta - 4 \delta^2)\length n^2} & \leq \Pr{X_{M,S} < (\delta -
3 \delta^2)(1-\delta)KL} \\
                                                  & \leq \Pr{X_j < (\delta - 2
\delta^2)|M| \text{ for some } 2 \leq j \leq L} \\
                                                  & \leq Le^{- 2 \delta ^4 |M|},
 \end{align*}
where the second inequality holds for large $n$ because the contribution from
$X_1$ is negligible as $L \to \infty$.

Let $A = \{x \in V(G): x \text{ is visited at least } (1-4\delta)n \text{ times by } W\}$.  If $|A| < (1 - \delta) n$, then either $\delta K$ of the $x_i$ are unbalanced, or there is set $S$ of $\delta n$ vertices and a multiset $M$ of $(1-\delta )K$
balanced vertices of $V(G)$ for which $X_{M,S} < (\delta - 4 \delta^2) \length n^2$.  Hence
\begin{align*}
 \Pr{|A| < (1 - \delta) n} & \leq \Pr{U \geq \delta K} + \sum_M \sum_S L e^{- 2 \delta ^4 |M|} \\
 & \leq 2 \delta + \binom {K+n-1} {n-1} \binom n {\delta n} L e^{- 2 \delta ^4 (1-\delta )K} \\
 & \leq 2 \delta + O(K)^n \cdot 2^n \cdot L \cdot e^{- 2 \delta ^4 (1-\delta )K} \\
 & \leq 2 \delta + \exp\!\left( O(n \log n) + O(n) + O(\log n) - 2 \delta ^4 (1-\delta)K \right) \\
 & \leq 3 \delta,
\end{align*}
for $n$ sufficiently large, since $K \gg n \log n$.
\end{proof}

We now have everything we need to complete the proof of \thm{thm:general-case}.

\begin{proof}[Proof of \thm{thm:general-case}]
We will show that, with probability $1-\Oe$, the graph $\GW$ is close to $\GL
{\length /\rho}$.  It then follows from \thm{thm:list-quasirandomness} that $\GW$
is $\Oe$-quasirandom with probability $1-\Oe$.

Since there are at most $2\epsilon n < \delta$ unbalanced vertices in $G$, by
\lem{hitting-large-balanced-set}, with probability at least $1 - 3\delta$, there
is a set $B$ of $(1-2\delta)n$ balanced vertices such that every $v \in B$ is
hit at least $(1-4\delta)\length n \geq (1 -
5\delta)\frac{\length}{\density}d(v)$ times by $W$.  This accounts for
$(1-2\delta)n \cdot (1-4\delta)\length n \geq (1-7 \delta)\length n^2$ of the
list entries examined, so $\GW$ differs from $\GL {\length /\rho}$ by at most
$14\delta\length n^2$ edges.  Since $\delta$ tends to $0$ with $\epsilon$, the
result follows.
\end{proof}

\section{Trees}
\label{sec:trees}

A \defn{homomorphism} from a graph $H$ to a graph $G$ is an edge-preserving map
$\phi : V(H) \to V(G)$.
A random walk can be viewed as a random homomorphism of a path; a natural
generalisation is to consider a random homomorphism of some other tree $T$
(sometimes called a \defn{tree-indexed random walk}).  Just as we traversed a
path in one direction, our trees will be rooted and we think of them as directed
`downwards', away from the root.  In this section we will explore to what extent
the methods of \sect{sec:general-case} can be applied in this more general
setting.

We generate a random homomorphism as follows.
Enumerate the vertices of $T$ as $v_0, v_1, \dotsc, v_k$ where, for each $j$,
$T[v_0,\dotsc,v_j]$ is a connected subtree of $T$ containing the root $v_0$. 
First choose $\phi(v_0)$ from a given initial distribution.  Then, at each stage
$j>0$,
let $u$ be the parent of $v_j$ in $T$ and choose $\phi(v_j)$ uniformly at random
from the
neighbours of $\phi(u)$.  All choices are made independently, and we can think
of these choices as being taken from the lists $L_v$ as before.

Suppose now that $G$ is an $\epsilon$-quasirandom graph on $n$ vertices.  Let
$\phi$ be a random homomorphism of a tree $T$ of size $\length n^2$ to $G$, and
let $\GT$ be the subgraph of $G$ consisting of the edges in the image of $\phi$.
 Is $\GT$ quasirandom with high
probability?  It is easy to see that in general the answer is no. For
example, let $G=K_n$ and let $T$ be an $n/2$-ary tree of depth $2$ (here $\length
= 1/4 +
o(1)$).  Then with high probability $\phi(T)$ contains a constant fraction of
the edges of $G$.
But all of these edges are incident on the neighbourhood of the
root, which has only $(1-e^{-1/2}+o(1))n$ vertices with high probability, so,
with high probability, $\GT$ is not quasirandom.

We seek conditions on $T$ such that we can apply the approach taken in
\sect{sec:general-case}
with minimal changes. The condition we give here imposes an upper bound
on the maximum degree of $T$.

We need an analogue of the second model for the construction of a random walk. 
Instead of breaking our path into many short paths, we break our tree into many
small edge-disjoint subtrees.

\begin{lemma} \label{lem:decompose-tree}
 Let $T$ be a rooted tree with $N$ edges and let $L \leq N$.  Then $T$ can be
written as an edge-disjoint union of rooted trees $R_1, \dotsc, R_K$, each of
size between $L$ and $3L$.
\end{lemma}

\begin{proof}
 Let $v$ be a vertex of $T$ furthest from the root such that $v$ has at least
$L$ descendants.  Then each branch of $T$ lying below $v$ has at most $L$ edges,
so some union of these branches has size between $L$ and $2L$; let this be
$R_1$.  We obtain $R_2, \dotsc, R_K$ similarly until there are less than $L$
edges of $T$ remaining, which we add to $R_K$.
\end{proof}

Write $\R = \{R_1, \dotsc, R_K\}$ for the corresponding set of abstract rooted
trees, up to isomorphism.  In an abuse of notation we use $R_i$ to refer to both
the specific subtree of $T$ and its isomorphism type.

It is convenient to number the $R_i$ such that $R_1 \cup \dotsb \cup R_j$ is a
subtree of $T$ containing the root for each $j$.  We can then describe the
second model for the construction of a random homomorphism as follows.  For each
$v \in V(G)$ and $R\in\R$, let $\Lambda_{v, R}$ be a list of independent random
homomorphisms from $R$ to $G$ that map the root of $R$ to $v$.  Choose a vertex
$v_1$ from the given distribution for the image of the root of $T$ and identify
$\phi(R_1)$ with the first entry from $\Lambda_{v_1,R_1}$.  (If $R_1$ has a
non-trivial automorphism group then there is a choice of identification of $R_1$
with the reference copy in $\R$.  The choice is unimportant provided the same
choice is made every time.)  Then at each stage $j$ we have already determined
the image $v_j$ of the root of $R_j$, and we identify $\phi(R_j)$ with the first
unused element of $\Lambda_{v_j, R_j}$.

Now let $T$ be a rooted tree with $\length n^2$ edges.  As before we want to show
that $T$ `visits' most vertices of $G$ about the right number of times.
We need to be careful here about what counts as a `visit': what we want to count
is the number of times an edge leaves a vertex, as that is the number of entries
of the corresponding list that will be examined.  So we say $\phi(T)$
\defn{visits} $x \in V(G)$ whenever $uv$ is an edge of $T$ with $u$ the parent
of $v$ and $\phi(u) = x$; the number of visits $\phi(T)$ makes to $x$ is the
number of edges $uv$ for which this occurs.

There are three places where the argument in the proof of
\lem{hitting-large-balanced-set} needs modification or additional details need
to be checked.

\begin{enumerate}
\item[(i)]
In the path case the edges (or vertices) of the blocks had a natural order and
the blocks were all the same size.  In the tree case we are free to choose a
labelling of the edges in each block, but the blocks might still have different
sizes: when we look at the $2L$th edge from each block, are there enough blocks
with $2L$ edges that Chernoff's inequality will give good concentration?
\item[(ii)]
In the path case the set of list entries examined was parameterised by multisets
of vertices of $G$.  In the tree case the set of list entries examined is
instead parameterised by multisets of pairs $(v, R)$ with $v \in V(G)$ and $R
\in \R$.  So the factor $\binom {K+n-1} {n-1}$ in the final sum needs to be
replaced by $\binom {K+n|\R|-1} {n|\R|-1}$, and we must restrict the size of
$\R$ to prevent this becoming too large.
\item[(iii)]
In the path case we had to ignore the first two vertices of each block as we
needed to take two steps before we had good information about the distribution
over vertices.  This was safe because the ignored vertices were only a $o(1)$
fraction of the total number of vertices.  In the tree case we must ignore the
edges whose start point is the root of the block or is a child of the root.  We
need to ensure that the number of ignored edges is at most a small fraction of
the total number of edges.
\end{enumerate}

Problem (i) is avoided by throwing away the small number of edges that receive a
label shared by few other edges. If we throw away all edges that receive a label
which is used less than $\epsilon n^2 / L^2$ times then the total number of
edges thrown away is less than $3 \epsilon n^2 / L$ as there are at most $3L$
edges in each block.

\begin{figure}[ht]
\begin{centering}
\includegraphics{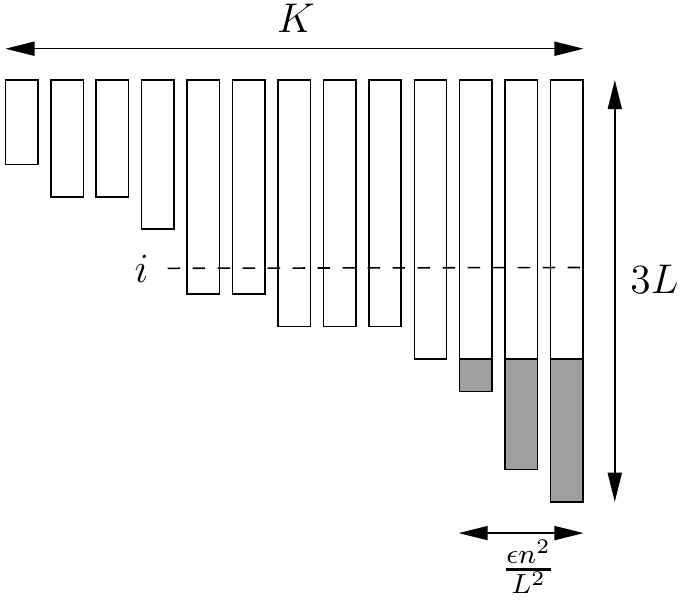}
\caption{Deleting a $o(1)$ fraction of the edges ensures that the remaining
labels $i$ are each used in a large number of blocks.}
\label{fig:pruning-blocks}
\end{centering}
\end{figure}

Problem (ii) is avoided by taking $L$ small: $L=\frac{\log n}{2\log 3}$
suffices.  Indeed, since the number of rooted trees on $L$ vertices is
$O((2.9955\ldots)^L)$ (see \cite{Otter}) and $\frac{\length n^2}{3L} \leq K \leq
\frac{\length n^2}{L}$, we have in this case that $n|\R| \ll n^{3/2} \ll K$, and
\[
 \binom {K+n|\R|-1} {n|\R|-1} \ll K^{n|\R|} \ll \exp\left(O(n^{3/2}\log
n)\right),
\]
which is small enough that it will not overpower the $e^{- cK}$-type decay.

Problem (iii) is avoided by having $\Delta^2$, the square of the maximum degree
of $T$ small (depending on the desired level of quasirandomness) compared to
$L$: so $\Delta$ can be as large as a small multiple of $\sqrt {\log n}$.

With these modifications to our earlier argument we obtain the following result.

\begin{theorem}
 \label{thm:random-homomorphism}
 Given $\length, \rho, \eta >0$ there exists $\epsilon, c > 0$ 
 such that the following holds.  Let $G$ be an $n$-vertex $\epsilon$-quasirandom
graph with $\rho \binom n 2$ edges, let $T$ be a rooted tree of size $\length n^2$
with maximum degree $\Delta \leq c \sqrt{\log n}$ and let $\phi$ be a
 random homomorphism from $T$ to $G$ such that the image of the root is
balanced. Then, with probability
 $1 - \Oe$, the subgraph $\GT$ of $G$ consisting of the edges of $\phi (T)$
is $\eta $-quasirandom with
 $(1 - e^{-2\length /\rho } + \Oe)\rho \binom n 2$ edges.
\end{theorem}

It would be interesting to know how large $\Delta (T)$ can be taken in
\thm{thm:random-homomorphism}. By 
the example at the start of this section we must have $\Delta (T)$ small
compared to $n$. Is this already enough?

\vspace{2em}
{\small \noindent
\textbf{Acknowledgements.} We would like to thank the organisers of the
\emph{Probabilistic methods in graph theory}
workshop at the University of Birmingham where we heard about this problem. We
would also like to thank Jan Hladk\'y
for some helpful discussions and for suggesting the extension of our results on
paths to more general trees.
}

\end{document}